\documentclass{article}
\usepackage[utf8]{inputenc}
\usepackage{amsmath, amssymb, mathtools}
\usepackage{mathrsfs}
\usepackage{algorithm, algorithmicx, algpseudocode}
\usepackage{tikz}
\usetikzlibrary{arrows, shapes, chains, scopes}
\usepackage{tcolorbox}
\usepackage{graphicx}
\usepackage{caption,subcaption}  
\usepackage{textcomp}
\usepackage{amsthm}
\usepackage[numbers, sort]{natbib}
\usepackage{xcolor}
\usepackage{authblk}

\newtheorem{theorem}{Theorem}[section]
\newtheorem{prop}{Proposition}[section]

\theoremstyle{remark}
\newtheorem{assumption}{Assumption}
\newtheorem{remark}{Remark}

\theoremstyle{definition}
\newtheorem{example}{Example}

\newcommand{\E}{\mathbb{E}}

\providecommand{\keywords}[1]{\textbf{\textit{Keywords:}} #1}

\title{Adaptive Gradient Descent for Optimal Control of Parabolic Equations with Random Parameters}

\author[1]{Yanzhao Cao\thanks{yzc0009@auburn.edu}}
\author[1]{Somak Das \thanks{szd0041@auburn.edu}}
\author[1]{Hans-Werner van Wyk \thanks{hzv0008@auburn.edu}}
\affil[1]{Department of Mathematics and Statistics, Auburn University, Auburn AL}

\date{\today}

\begin{document}

\maketitle

\begin{abstract}
In this paper we extend the adaptive gradient descent (AdaGrad) algorithm to the optimal distributed control of parabolic partial differential equations with uncertain parameters. This stochastic optimization method achieves an improved convergence rate through adaptive scaling of the gradient step size. We prove the convergence of the algorithm for this infinite dimensional problem under suitable regularity, convexity, and finite variance conditions, and relate these to verifiable properties of the underlying system parameters. Finally, we apply our algorithm to the optimal thermal regulation of lithium battery systems under uncertain loads.  
\end{abstract}


\section{Introduction}

Stochastic optimization algorithms have increasingly found a use in the design of deterministic regulators for uncertain systems. Such problems arise in open loop control systems exhibiting statistical variations that cannot be observed by the controller. The control must consequently be designed to be robust in light of predicted uncertainties to ensure a desired statistical behavior of the system, as encoded by an appropriate risk function. In this paper we measure risk in terms of the mean squared deviation of the controlled state from a desired reference state, but other risk functions, such as Conditional Value at Risk \cite{Rockafellar2002}, are also possible. In this work we formulate and  analyze the adaptive gradient descent (AdaGrad) method for an optimal control problem constrained by a partial differential equation with uncertain coefficients.  

We focus on the distributed control of uncertain parabolic systems. These problems arise naturally in the thermal regulation of lithium battery systems \cite{Wang2017, Guo2010}, which show great promise for building grid-level energy storage systems because of their high energy and power density, low discharge rate, and increasingly low production cost. The variability of operating conditions, of the manufacturing process, and of the degradation of batteries over their life cycle gives rise to uncertainties in the thermal properties of such systems. Their effect on material properties such as local resistivity, can be quantified through various methods, including empirical testing, and the use of battery degradation models \cite{Xu2018,Liu2015a}. Lumped parameter and equivalent circuits models, commonly used in the design of efficient battery management systems \cite{Bergveld2002}, can fail to detect internal thermal conditions, especially under irregular operating conditions \cite{Forgez2010}. This suggests the use of an optimal controller, designed in response to statistical variations to ensure a desired average overall temperature over a range of conditions and states of degradation.

In mathematical terms, the design of an optimal control for uncertain systems can be formulated as a deterministic, infinite dimensional optimization problem  whose cost function takes the form of a stochastic integral. Stochastic gradient (SG) algorithms are line search methods in which the gradient of the risk function is replaced by random gradient samples, resulting in iterations that are much cheaper to compute than those obtained by a full approximation of the gradient. These stochastic optimization methods, originally conceived in \cite{Robbins1951}, and fundamental in the development machine learning algorithms, have recently garnered attention in the context of infinite dimensional optimization (see e.g. \cite{Martin2018a, Geiersbach2019}). They are particularly well-suited to large scale problems in which the risk function is strongly convex and the underlying uncertainty is sufficiently complex to warrant the use of Monte Carlo sampling in approximating the stochastic integrals. It can be shown (see e.g. \cite{Bottou2018}) that for strongly convex risk functions whose  sample gradients are Lipschitz continuous with bounded variance, the stochastic gradient method with appropriately chosen step-sizes converges at the rate $O(1/j)$, where $j$ is the number of gradient evaluations. The convergence rate is optimal among first order methods, according to the complexity bounds established in \cite{Agarwal2009}. In comparison, a gradient-based deterministic optimization scheme, coupled with standard Monte Carlo approximation, has a convergence rate of $O(1/\sqrt{j})$. The cost-reduction offered by SG iterations is especially pertinent in the context of optimal control, where the evaluation of the sample gradient involves the numerical approximation of two partial differential equations, namely the state and the adjoint equations (see Section \ref{section:problem_setting}). 

While the SG iteration converges for a range of predetermined stepsizes, its convergence rate can vary widely. Moreover, the step size rule that guarantees the optimal convergence rate depends on the gradient's Lipschitz constant, its variance, and its strong convexity parameter, all of which are difficult to estimate in general. This practical shortcoming has led to investigations into adaptive step size rules that use information obtained during the iteration to adjust the step sizes on the fly. The Adaptive Gradient (AdaGrad) method, developed concurrently in \cite{Duchi2011} and \cite{McMahan2010}, scales the step size by the cumulative sum of gradients sampled thus far. Despite the widespread use of AdaGrad and its extensions, such as the Root Mean Square Propagation (RMSProp), or the Adaptive Moment Estimation (Adam) algorithms \cite{Hinton2012,Kingma2014}, theoretical insight into its robustness and convergence has remained elusive until recently \cite{ward19a, Li2019a}, even in the finite dimensional case. In \cite{Li2019a}, the authors prove the convergence of the AdaGrad method for both strongly convex and general functions over finite dimensional parameter spaces, obtaining convergence rates in the expected optimality gap that interpolate between $1/j$ and $1/\sqrt{j}$. In their proof the authors require a bound on the initial step size parameter that involves the risk function's Lipschitz constant. In \cite{ward19a}, the authors prove a weaker form of convergence for a general risk function without imposing any conditions on the step size in terms of the underlying problem parameters.

This paper aims to extend the AdaGrad method to infinite dimensional distributed control systems constrained by parabolic PDEs with uncertainties. In Section \ref{section:problem_setting} we outline the optimal control problem and introduce the AdaGrad algorithm. In Section \ref{section:convergence} we establish convergence of the AdaGrad algorithm and relate its requirements on the risk function's regularity, convexity, and finite variance to verifiable properties of the system's uncertain parameters. The numerical experiments in Section \ref{section:numerical_experiments} illustrate the algorithm's theoretical properties derived in previous section, as well its appplication to a thermal regulation problem. In Section \ref{section:conclusion}, we offer concluding remarks.

\section{Problem Setting}
\label{section:problem_setting}

Let $(\Omega, \mathcal F, \mathbb{P})$ be a complete probability space encoding the uncertainties in our system, $D \subset \mathbb{R}^d$, $d=1,2,3$, be a physical domain with boundary $\partial D$, and $T>0$ be some terminal time. The system's state $y:D\times [0,T]\times \Omega \rightarrow \mathbb{R}$ is then defined as the random field satisfying 
\begin{equation}
\label{eq:state}
	\begin{cases}
    \frac{dy}{dt} - \nabla (a \nabla y) = g + u, & x \in D,\ t \in [0,T], \\
    y = 0, & x \in \partial D, \ t\in [0,T],  \\
    y(\cdot,0) = y_0, & x \in D 
	\end{cases}
\end{equation}
almost surely (a.s.)\! on $\Omega$. For convenience, we define the differential operator $\mathcal L = \frac{d}{dt} - \nabla (a\nabla \cdot)$. Uncertainties in the system can arise from the diffusion coefficient $a$, the forcing term $g$, or the initial condition $y_0$. The deterministic function $u \in U := L^2(D\times[0,T])$ represents the distributed control to be determined through optimization. To ensure the state equation's well-posedness, we make the following assumptions.

\begin{assumption}
\label{ass:well_posedness}
The diffusion coefficient $a\in L^\infty([0,T]\times D \times \Omega)$ satisfies the coercivity condition
\begin{equation}
0 < a_{\min} \leq a(x,t,\omega) \leq a_{\max} < \infty, \qquad x \in D, \ t \in [0,T], \ \omega \in \Omega.
\label{eq:coercivity}
\end{equation}
for constants $a_{\min}, a_{\max} \in \mathbb{R}$. The forcing term $g$ is a square integrable mapping $g:\Omega \rightarrow L^2([0,T],H^{-1}(D))$, i.e. $g \in L^2(\Omega, L^2([0,T],H^{-1}(D))$, and the initial condition $y_0 \in L^2(D \times \Omega)$.
\end{assumption} 

As a consequence, Equation \eqref{eq:state} has a unique solution $y$ for every control $u$ \cite{Babuska2007}. The assumption on the uniform coercivity can be relaxed somewhat to allow for lognormal diffusion coefficients, see \cite{Gittelson2010}. In the following, we will often find it useful to refer to $y$ in terms of its dependence on various subsets of $x,t, \omega$, and the control $u$, e.g. $y(u,\omega)$ or $y(x,t,\omega)$. Throughout the paper, we will use $\langle \cdot, \cdot \rangle$ and $\|\cdot \|$ to refer to the inner product and norm associated with $L^2([0,T]\times D)$, i.e. for $v \in L^2([0,T]\times D)$,
\[
\|v\|^2 = \int_0^T \int_D |v(x,t)|^2 dx\;dt.
\]

\subsection{The Optimal Control Problem}

In optimal control we seek a function $u$ that steers the corresponding state $y(u)$ to track a desired reference solution $y_d$. This function is often deterministic, but here we need only require $y_d \in L^2([0,T]\times D \times \Omega)$, allowing us to encode the state's desired behavior on a statistical level. We measure the deviation of $y(u)$ from $y_d$ in a mean-squared sense over all possible state realizations. Furthermore, we constrain the control $u$ to lie within the admissible set 
\[
U_{ad} = \{u \in U: \|u\| \leq u_{\max}\}, 
\]
for a given constant $0<u_{\max} \leq \infty$. The optimal control problem can thus be stated as
\begin{equation}
\label{eq:optimal_control}
\min_{u\in  U_{ad}} F(u),
\end{equation}
where the objective- or risk function is 
\[
F(u) = \E[f(u)] = \int_\Omega f(u,\omega)d\mathbb{P}(\omega),
\]
and 
\[
f(u,\omega) = \frac{1}{2}\|y(u,\omega)-y_d(\omega)\|^2  + \frac{\alpha}{2}\|u\|^2.
\] 

To establish well-posedness of Problem \eqref{eq:optimal_control}, i.e. the existence and uniqueness of an optimal control, we note that the risk function $F$ can be readily decomposed into
\[
F(u) = \pi(u,u) - 2L(u) + C, 
\] 
where
\[
\pi(u,v) = \frac{1}{2}\mathbb{E}\left[\langle y(u)-y_0, y(v)-y_0\rangle\right] + \frac{\alpha}{2}\mathbb{E}\left[\langle u,v \rangle \right]
\]
is a continuous, coercive bilinear form, 
\[
L(u) = \mathbb{E}\left[ \langle y_d-y_0, y(u)-y_0\rangle \right], \ \text{and} \\ 
\]
is a bounded linear form, $C$ is constant in $u$. The result then follows directly from Chapter 1, Theorem 1.1 in \cite{Lions1971}. In fact, it will be shown in Proposition \ref{prop:strong_convexity} that $f(u,\omega)$ is strongly convex a.s. on $\Omega$, which has strong implications on the convergence rate of the AdaGrad method (Algorithm \ref{alg:adagrad}), introduced in Section \ref{subsection:adagrad}.

\subsection{The AdaGrad Algorithm} 
\label{subsection:adagrad}

While Problem \eqref{eq:optimal_control} is ostensibly deterministic, its risk function $F(u)$ is a statistic associated with the uncertain system \eqref{eq:state}. For a given initial guess $u_0$, consider the stochastic gradient iteration
\begin{equation}
\label{eq:sg_iteration}
u_{j+1} = u_j - \eta_j\nabla f_j(u_j),
\end{equation}
where the stochastic gradients $\{\nabla f_j(u_j)\}_{j=1}^\infty$ are idependent, identically distributed samples of $\nabla f(u_j,\omega)$ satisfying $\mathbb{E}[\nabla f_j(u_j)]=F(u_j)$, and $\{\eta_j\}_{j=1}^{\infty}$ is a sequence of positive step sizes. 

\begin{remark}
In cases when the risk function is determined empirically, i.e. by testing of sampling, when the system's statistical distribution is determined empirically, through analysis of test samples, the risk function may take the form
\[
F(u) = \frac{1}{N} \sum_{i=1}^N f_i(u)
\]
sample gradients can be obtained as $\nabla f_{i_j}(u)$, where $i_j$ is drawn uniformly from the set $\{1,\hdots, N\}$ (see \cite{Bottou2018}).
\end{remark}
The established convergence analysis of the SGD iteration \eqref{eq:sg_iteration} for non-convex smooth functions relies on specific conditions on the positive step-sizes $\eta_j$, the most well-known of which is the Robbins-Monro condition (see \cite{Robbins1951}) requiring $\{\eta_j\}_{j=1}^{\infty}$ to be a deterministic sequence of positive numbers that satisfy
\begin{equation}
\label{robbins}
\sum_{j=1}^{\infty} \eta_j = \infty \hspace{5mm} \text{and} \hspace{5mm} \sum_{j=1}^{\infty} \eta_j^2 < \infty.
\end{equation}

A convenient choice of $\eta_j$ is that of a diminishing function of the iteration count, decreasing at the rate $\eta_j= O(\frac{1}{j})$ as $j\rightarrow \infty$. 
As an improvement to this SGD algorithm we apply the Adaptive Gradient Descent algorithm that uses the following adaptive step-size, 
\begin{equation}
\label{eq:adagrad_step size}
\eta_j = \frac{\eta}{\sqrt[]{b_0^2+ \sum\limits_{k=1}^{j-1} \|\nabla f_{k}(u_k)\|^2}}
\end{equation}
where $\eta >0$ and $b_0>0$ are constants. Other versions of AdaGrad scale componentwise. Other variations include the current gradient value (this can lead to steps that are not descent directions in expectation \cite{Li2019a}). 

For the optimal control problem \eqref{eq:optimal_control}, it is a standard result (see \cite{Lions1971}) that $f(u)$ is Fr\'echet differentiable and that the sample gradients $\nabla f(u)$ can be computed as
\begin{equation}
\label{eq:gradient}
\nabla f(u)=p + \alpha u,
\end{equation}
where $p$ is the adjoint state, satisfying 
\begin{equation}
\label{eq:adjoint}
\begin{cases}
-\frac{dp}{dt} - \nabla\cdot (a \nabla p) = y - y_d,  & x \in D, \ t \in [0,T],\\
p(x,t) = 0, & x \in \partial D, \ t \in [0,T], \\
p(x,T) = 0, & x \in D. 
\end{cases}
\end{equation}
Indeed, the directional derivative $D[f(u)](v)$ of $f$ at $u\in U$ in any direction $v \in U$ is given by
\[
D[f(u)](v) = \langle y-y_d, s(v) \rangle + \alpha \langle u, v\rangle,
\]
where $s(v)$ satisfies the sensitivity equation
\begin{equation}
\label{eq:sensitivity}
\begin{cases}
\frac{ds}{dt} - \nabla\cdot (a \nabla s) =v,  & x \in D, \ t \in [0,T],\\
s(x,t) = 0, & x \in \partial D, \ t \in [0,T], \\
s(x,0) = 0, & x \in D. 
\end{cases}
\end{equation}
Denoting by $\mathcal L^* = -\frac{d}{dt} - \nabla (a\nabla \cdot)$ the formal adjoint of $\mathcal L$, we can use Equation \eqref{eq:adjoint}, integration by parts, and Equation \eqref{eq:sensitivity} to write
\begin{align*}
D[f(u)](v) &= \langle y-y_d, s(v) \rangle + \alpha \langle u, v\rangle = \langle \mathcal L^*p, s(v) \rangle + \alpha \langle u, v\rangle \\
&= \langle p, \mathcal L s(v) \rangle + \alpha \langle u, v\rangle = \langle p, v\rangle + \alpha \langle u, v \rangle,   
\end{align*}
from which Formula \eqref{eq:gradient} follows. The AdaGrad algorithm for Problem \eqref{eq:optimal_control} can thus be summarized as follows.
\begin{algorithm}[H]
\caption{AdaGrad for Problem \eqref{eq:optimal_control}}\label{alg:adagrad}
\begin{algorithmic}[1]
\State Initialize $u_0$, $b_0$, $\eta$
\For{$j=0,1,\hdots$}
\State $y_j \gets$ solution of sample primal system \eqref{eq:state}
\State $p_j \gets$ solution of sample adjoint system \eqref{eq:adjoint}
\State $\nabla f_j(u_j) = p_j + \alpha u_j$
\State $u_{j+1} = u_j - \frac{\eta}{b_j} \nabla f_j(u_j) $ 
\State $b_{j+1}^2 = b_j^2 + \|\nabla f_j(u_j)\|^2$
\EndFor  
\end{algorithmic}
\end{algorithm}

\section{Convergence Analysis}\label{section:convergence}

In this section we establish some theoretical properties of the AdaGrad algorithm in the context of the optimal control problem \eqref{eq:optimal_control}. Theorem \ref{thm:convergence} proves its convergence under suitable conditions. To this end we show that the problem's risk function is strongly convex (Proposition \ref{prop:strong_convexity}) and Lipschitz continuous (Proposition \ref{prop:gradient_lipschitz}). Moreover, in Section \ref{subsection:finite_variance} we relate the algorithm's finite variance requirement to the statistical properties of the underlying uncertain inputs. Our convergence analysis is based on that of \cite{Li2019a} whose authors establish convergence for convex risk functions over finite-dimensional parameter space, whose sample gradient functions $\nabla f$  are Lipschitz continuous, unbiased, and uniformly bounded over $u$ and $\omega \in \Omega$.

\subsection{Convexity and Smoothness}
\begin{prop}
\label{prop:strong_convexity}
The mapping $f(\cdot, \omega) :U \rightarrow \mathbb{R}$ is strongly convex a.s. on $\Omega$. Specifically,
\begin{equation}
\label{eq:convexity}
\langle \nabla f(v,\omega) - \nabla f(u,\omega), v-u \rangle \geq \alpha \| v - u \|^2 \qquad \text{ a.s.\! on } \Omega.
\end{equation}
\end{prop}
\begin{proof}
Since the result holds a.s.\! on $\Omega$, we find it notationally convenient not to write $\omega$ explicitly in this proof. Specifically, let $\omega \in \Omega$ be a fixed system realization and let $y(u)$, $p(u)$, and $f(u)$ be the associated state, adjoint variable, and sample cost when subjected to control $u\in U$. Using the form of the gradient given by Equation \eqref{eq:gradient}, we have that for $u,v \in U$, 
\begin{equation}
\label{eq:difference_of_gradients}
\langle \nabla f(v) - \nabla f(u), v-u \rangle = \alpha \| v-u \|^2 + \langle p(v)-p(u), v-u\rangle.
\end{equation}
To prove the proposition, it therefore suffices to show ${\langle p(v)-p(u), v-u\rangle \geq 0}$. To this end, note that subtracting Equation \eqref{eq:state} with control $u$ from that with control $v$ gives
\begin{equation}
\label{eq:difference_of_state}
\mathcal L (y(v) - y(u)) = v-u. 
\end{equation}
Similarly, 
\begin{equation}
\label{eq:difference_of_adjoint}
\mathcal L^* (p(v) - p(u)) = y(v) - y(u).
\end{equation}
Multiplying Equation \eqref{eq:difference_of_state} by $p(v)-p(u)$ and integrating over $D$ and $[0,T]$ gives 
\begin{align*}
\langle \mathcal L(y(v)-y(u)), p(v)-p(u) \rangle = \langle v-u, p(v)-p(u)\rangle,
\end{align*}
which, through integration by parts and Equation \eqref{eq:difference_of_adjoint}, takes the form
\begin{align*}
\langle p(v)-p(u), v-u \rangle &= \langle \mathcal L(y(v)-y(u)), p(v)-p(u) \rangle\\
 &= \langle y(v)-y(u), \mathcal L^*(p(v)-p(u)) \rangle = \| y(v)-y(u) \|^2 \geq 0. 
\end{align*}
This proves Inequality \eqref{eq:convexity}, by virtue of Equation \eqref{eq:difference_of_gradients}.
\end{proof}

\begin{prop}
\label{prop:gradient_lipschitz}
For $f$ defined in \eqref{eq:gradient}, $u,v \in U$, and for $\omega \in \Omega$, we have
\begin{equation}\label{eq:gradient_lipschitz}
\|\nabla f(v) - \nabla f(u)\| \leq M \|v-u\|,
\end{equation}
where $M = \alpha + \frac{C_p^4}{a_{\min}^2}$.
\end{prop}

\begin{proof}
As before we refrain from writing $f,y,$ and $p$ explicitly in terms of the random state $\omega \in \Omega$. Recall Equation \eqref{eq:difference_of_gradients}
\[
\nabla f(v) - \nabla f(u) = \alpha(v - u) + p(v) - p(u),
\]
for a given $\omega \in \Omega$, and controls $u, v \in U$. To bound the difference in sample gradients, it therefore suffices to bound the difference in adjoint variables. Multiplying Equation \eqref{eq:difference_of_adjoint} by $p(v)-p(u)$, integrating over $D$ and $[0,T]$, and using Green's theorem for the diffusion term, as well as the chain rule for the time derivative yields
\begin{align}
&\langle y(v)-y(u), p(v)-p(u)\rangle = \langle \mathcal L^*(p(v)-p(u)), p(v)-p(u)\rangle \nonumber\\
=& -\int_0^T \int_D \frac{d}{dt}\|p(v)-p(u)\|^2 dx\; dt + \int_0^T \int_D a (\nabla p(v)-\nabla(p(u))^2 dx\; dt. \label{eq:cross_term}
\end{align}
The fundamental theorem of calculus and the terminal condition for the adjoint equation \eqref{eq:adjoint} together imply
\[
-\int_0^T \int_D \frac{d}{dt}\|p(x,t,v)-p(x,t,u)\|^2 dx\; dt = \int_D \|p(x,0,v)-p(x,0,u)\|^2 dx \geq 0.
\]
Moreover, the coercivity condition \eqref{eq:coercivity}, used in conjunction with the Poincar\'e inequality, implies
\[
\int_0^T \int_D a (\nabla p(v)-\nabla(p(u))^2 dx\; dt \geq a_{\min} \|\nabla p(v)-\nabla p(u)\|^2 \geq \frac{a_{\min}}{C_p^2} \|p(v)-p(u)\|^2,  
\]
where $C_p$ is the appropriate Poincar\'e constant. Equation \eqref{eq:cross_term} and the Cauchy-Schwartz inequality thus lead to
\[
\frac{a_{\min}}{C_p^2} \|p(v)-p(u)\|^2 \leq \langle y(v)-y(u), p(v)-p(u)\rangle \leq \|y(v)-y(u)\|\|p(v)-p(u)\|,
\]
and hence 
\begin{equation}
\|p(v)-p(u)\|\leq \frac{C_p^2}{a_{\min}} \|y(v)-y(u)\|.
\end{equation}
By applying analogous arguments to the state equation \eqref{eq:state}, we can similarly bound
\[
\|y(v)-y(u)\| \leq \frac{C_p^2}{a_{\min}} \|v-u\|,
\]
so that 
\[
\|p(v)-p(u)\| \leq \frac{C_p^4}{a_{\min}^2} \|v-u\|.
\]
Therefore 
\[
\|\nabla f(v)-\nabla f(u)\| \leq \|p(v)-p(u)\| + \alpha \|v-u\| \leq \frac{C_p^4}{a_{\min}^2} \|v-u\| + \alpha \|v-u\|, 
\]
yielding the bound \eqref{eq:gradient_lipschitz}.
\end{proof}

\subsection{Finite Variance}\label{subsection:finite_variance}
In addition to strong convexity and Lipschitz continuity, shown in Propositions \ref{prop:strong_convexity} and \ref{prop:gradient_lipschitz} respectively, the convergence proof of the AdaGrad algorithm requires a strengthened finite variance condition on the sample gradients. It is commonly enforced \cite{Nemirovski2009, Li2019a} by assuming there is a constant $\sigma^2>0$ so that  
\begin{equation}
\label{eq:finite_variance}
\E\left[\exp\left(\frac{\|\nabla f(u) - \nabla F(u)\|^2}{\sigma^2}\right)\right] \leq \exp(1), \qquad \text{for all } u\in U.
\end{equation}
Through the use of Jensen's inequality and conditional expectation, Inequality \eqref{eq:finite_variance} can be readily shown (see e.g. \cite{Li2019a}) to imply
\begin{equation}
\label{eq:finite_variance_max}
\E\left[\max_{1\leq k \leq j} \|\nabla f_k(u_k) - \nabla F(u_k)\|^2\right] \leq \sigma^2(1 + \ln(j)), \qquad j=1,2,\hdots,
\end{equation}
in addition to the standard variance bound
\begin{equation}
\label{eq:finite_variance_classic}
\E\left[\|\nabla f_j(u_j) - \nabla F(u_j)\|^2\right] \leq \sigma^2, \qquad j=1,2,\hdots.
\end{equation}
Finite variance assumptions, such as Inequality \eqref{eq:finite_variance}, are commonly made in convergence analyses of stochastic optimization methods and are, strictly speaking, only required to hold for the iterations $u_j$. Nevertheless, they are generally not verifiable independently of the iteration.   
The following proposition helps frame requirement \eqref{eq:finite_variance} in terms of statistical paramaters underlying the control problem \eqref{eq:optimal_control}.

\begin{prop}\label{prop:bnd_grad_by_pars}
For any $\omega \in \Omega $ and any fixed $u \in U$,
\begin{multline}
\label{eq:gradient_par_bnd}
\|\nabla f(u,\omega)\| \leq \left(\alpha + \frac{C_p^4}{a_{\min}^2}\right) \|u\| + \frac{C_p^4}{a_{\min}^2}\|g(\omega)\| + \frac{C_p^4}{2a_{\min}^2}\|y_0(\omega)\| + \frac{C_p^2}{a_{\min}}\|y_d(\omega)\|.
\end{multline}
\end{prop}
\begin{proof}
Recall from Equation \eqref{eq:gradient} that 
\[
\|\nabla f(u,\omega)\| = \|\alpha u + p(u,\omega)\|\leq \alpha \|u\|+\|p(u,\omega)\|,
\]
for any $\omega \in \Omega$, where $p=p(u,\omega)$ satisfies the adjoint system given by Equation \eqref{eq:adjoint}. To bound $\|p\|$, we multiply on both sides of Equation \eqref{eq:adjoint} by $p$ and integrate over $[0,T]$ and $D$, leading to 
\begin{equation}
\label{eq:p_integrated}
-\int_0^T\int_D p_t p \; dx\;dt - \int_0^T \int_D \nabla \cdot (a \nabla p) p\; dx\;dt = \int_0^T \int_D (y - y_d)p \;dx\;dt.
\end{equation}

Using the chain rule, the fundamental theorem of calculus, and the terminal condition for $p$, we obtain
\begin{equation}
\label{eq:p_time_bnd}
-\int_0^T\int_D p_t p \; dx\;dt = -\int_0^T \int_D \frac{1}{2} \frac{d}{dt}p^2 \;dx \;dt = -\frac{1}{2}\int_D p(x,0)^2 dx\;dt  \leq 0.
\end{equation}
Moreover, by Green's theorem
\begin{equation}
\label{eq:p_diffusion_bnd}
- \int_0^T \int_D \nabla \cdot (a \nabla p) p\; dx\;dt = \int_0^T \int_D a \nabla p\cdot \nabla p \;dx\;dt \geq a_{\min} \|\nabla p\|^2.
\end{equation}
Substituting Equations \eqref{eq:p_time_bnd} and \eqref{eq:p_diffusion_bnd} into \eqref{eq:p_integrated}, rearranging terms, and using the Poincar\'e inequality then results in 
\begin{equation}
\label{eq:p_in_terms_of_y}
\|p\|\leq \frac{C_p^2}{a_{\min}} \|y-y_d\| \leq \frac{C_p^2}{a_{\min}}\left(\|y\|+ \|y_d\|\right).
\end{equation}
To bound $\|y\|$, we multiply the state equation \eqref{eq:state} by $y$ on both sides and use analogous arguments to those above to obtain
\[
a_{\min}\|\nabla y\|^2 \leq (\|g\| + \|u\|)\|y\| + \frac{1}{2}\|y_0\|^2.
\] 
Since 
\[
\|y_0\| = \sqrt{\int_D y_0(x)^2 \;dx} \leq \sqrt{\int_0^T \int_D y(x,t)^2 dx\;dt} = \|y\|,
\]
we have, by virtue of Poincar\'e's inequality, that 
\[
\|y\| \leq \frac{C_p^2}{a_{\min}}\left(\|g\| + \|u\| + \frac{1}{2}\|y_0\|\right).
\]
Substituting this inequality into \eqref{eq:p_in_terms_of_y} establishes the estimate in \eqref{eq:gradient_par_bnd}. 
\end{proof}

Note that the upper bound \eqref{eq:gradient_par_bnd} involves the norm $\|u\|$ of the control. In fact, a derivation similar to that used to establish Proposition \ref{prop:bnd_grad_by_pars} leads to an upper bound for $\|\nabla f(u,\omega)-\nabla F(u)\|$, for $\omega \in \Omega$, that also depends on $\|u\|$, as long as the diffusivity $a$ is stochastic. We therefore find it necessary to restrict the admissible set $U_{ad}$ to be bounded, i.e. $u_{\max}<\infty$. To ensure that AdaGrad iterations remain feasible, we make the following assumption. 
\begin{assumption}
\label{ass:iterates_remain_bounded}
Assume that the samples and step sizes in the AdaGrad iteration are chosen so that $u_j \in U_{ad}$ for $j=1,2,\hdots$. 
\end{assumption}
This condition can be enforced for example by rejecting steps that lie outside $U_{ad}$. To be sure, $u_{\max}$ may be chosen large enough to make such interventions unlikely as well as to ensure that the optimal control $u^*$ lies in the interior of $U_{ad}$. In fact, by the optimality of $u^*$ and form of the risk function $F(u)$, we have 
\[
\frac{\alpha}{2} \|u^*\|^2 \leq F(u^*) \leq F(u_0),
\]
so that $\|u^*\|<u_{\max}$ whenever $u_{\max} > \sqrt{\frac{2}{\alpha}F(u_0)}$ for any initial condition $u_0$.

We will use the estimates of Proposition \ref{prop:bnd_grad_by_pars} to establish conditions on the uncertain parameters that ensure the strengthened bounded variance condition \eqref{eq:finite_variance}. To simplify the bound \eqref{eq:gradient_par_bnd}, we estimate $\|u\|$ by $u_{\max}$ and bound the constants by their upper bound $K = \max\left\{\alpha + \frac{C_p^4}{a_{\min}^2}, \frac{C_p^2}{a_{\min}}\right\}$, yielding
\begin{equation}
\label{eq:gradient_par_bnd_simple}
\|\nabla f\| \leq K(u_{\max}+ \|g\| + \|y_d\| + \|y_0\|).
\end{equation} 
For any $\omega \in \Omega$, repeated use of Jensen's inequality gives
\begin{multline*}
\E\left[\exp\left(\frac{\|\nabla f(u,\omega)-\nabla F(u)\|^2}{\sigma^2}\right)\right] \\ \leq \E\left[\exp\left(\frac{2}{\sigma^2}\left(\|\nabla f(u,\omega)\|^2+\E\left[\|\nabla f(u)\|^2\right]\right)\right)\right],
\end{multline*}
where \eqref{eq:gradient_par_bnd_simple} and Assumption \eqref{ass:well_posedness} imply
\begin{align*}
\E\left[\|\nabla f(u)\|^2\right] \leq 4 K^2 \E\left[u_{\max}^2+\|g\|^2 +\|y_d\|^2 + \|y_0\|^2\right] <\infty.
\end{align*}
In light of these bounds, we can now formulate an assumption on the system parameters $g$, $y_d$, and $y_0$ that, together with Assumptions \ref{ass:well_posedness} and \ref{ass:iterates_remain_bounded}, implies the strengthened bounded variance condition \eqref{eq:finite_variance}.
\begin{assumption}
\label{ass:finite_variance}
Assume that there exists a $\sigma^2>0$, so that 
\begin{multline}
\E\left[\exp\left(\frac{8K^2}{\sigma^2}\left(\|g\|^2 + \|y_d\|^2 + \|y_0\|^2 +\E\left[\|g\|^2 +\|y_d\|^2 + \|y_0\|^2\right]\right)\right)\right] \\
\leq \exp\left(1-\frac{16 K^2}{\sigma^2}u_{\max}^2\right).
\end{multline}
\end{assumption}

With the cost function $f$ satisfying Propositions 2.1-2.4 and from \textbf{Theorem 3} in  \cite{Li2019a} we can state that:
\begin{theorem}\label{thm:convergence}
For stepsizes given by \eqref{eq:adagrad_step size}, where $\eta, b_0 >0$ and $4 \eta M < \sqrt{b_0}$, and under Assumptions \ref{ass:well_posedness}--\ref{ass:finite_variance}, the iterates of the AdaGrad algorithm satisfy the following bound
\begin{equation}
\label{convg_bound}
\mathbb{E}\left[\sqrt{F(\bar{u}_n)-F(u^*)}\right] \leq \frac{1}{\sqrt{n}} \max \left(\gamma \sqrt{M} , (b_0 + n \sigma^2)^\frac{1}{4} \sqrt{\gamma}\right),
\end{equation}
where 
\[
\bar{u}_n = \frac{1}{n}\sum_{j=1}^n u_j, \qquad \text{ and } \qquad \gamma = O\left(\frac{1+ \eta^2 \ln n}{\eta (1-\frac{4 \eta M}{\sqrt{b_0}})} \right)
\]
\end{theorem}

\begin{remark}\label{rem:markov_inequality}
The convergence bound ensures, via Markov's inequality, that the optimality gap $F(\bar u_n)-F(u^*)$ will satisfy
\[
F(\bar{u}_n)-F(u^*) \leq \frac{1}{\delta}\left(\frac{1}{n}\max\left\{\gamma^2 M, (b_0+\sigma^2 n)^{\frac{1}{2}}\gamma\right\}\right)
\]
with a probability of at least $1-\delta$. This bound highlights the effect of the variance on the method's convergence rate, reducing it from $O\left(\frac{1}{\sqrt{n}}\right)$ to $O\left(\frac{1}{n}\right)$.
\end{remark}

\begin{remark}
Note that, while the step size rule $\{\eta_j\}_{j=1}^\infty$ does not depend on the variance parameter $\sigma^2$, it is constrained by the Lipschitz constant $M$. In \cite{ward19a} the authors provide a proof for a weaker form of convergence without any such dependence. 
\end{remark}

\begin{proof}
For completeness, we give a brief outline of the proof. A more detailed discussion can be found in \cite{Li2019a}.
Let $\delta F_j = F(u_j) - F(u^*)\geq 0$ for $j=0,1,2,\hdots$ be the optimality gap at the jth iteration. The convexity of $F$ implies, by virtue of Jensen's inequality, that
\begin{align*}
\E\left[\sqrt{F(\bar u_n)-F(u^*)}\right] &= \E\left[\sqrt{F\left(\frac{1}{n}\sum_{j=1}^n u_j\right)-F(u^*)}\;\right]\\
&\leq \E\left[\sqrt{\frac{1}{n}\sum_{j=1}^n F(u_j) - F(u^*)}\;\right] = \frac{1}{\sqrt{n}} \E\left[\sqrt{\sum_{j=1}^n\delta F_j}\;\right]
\end{align*}
Moreover, the fact that $\eta_0\geq \eta_1 \hdots \geq \eta_n > 0$, together with H\"older's inequality, justify
\begin{equation}
\label{eq:bnd_df}
\E\left[\sqrt{\sum_{j=1}^n \delta F_j}\;\right] \leq \E\left[\sqrt{\frac{1}{\eta_n}\sum_{j=1}^n \eta_j\delta F_j}\;\right] \leq \left(\E\left[\frac{1}{\eta_n}\right]\right)^{\frac{1}{2}}\left(\E\left[\sum_{j=1}^n \eta_j \delta F_j\right]\right)^{\frac{1}{2}}.
\end{equation} 
To complete the proof, it now remains to bound $\E\left[\frac{1}{\eta_n}\right]$ and $\E\left[\sum_{j=1}^n \eta_j \delta F_j\right]$.
From the definition of the step size, we have 
\begin{align*}
\frac{1}{\eta_n} &= \frac{1}{\eta}\sqrt{b_0^2 + \sum_{j=1}^{n-1} \|\nabla f_j(u_j)\|^2} \\
&\leq \frac{1}{\eta}\sqrt{b_0^2 + 2\sum_{j=1}^{n-1}\left( \|\nabla f_j(u_j)-\nabla F(u_j)\|^2 + \|\nabla F(u_j)\|^2\right)}.
\end{align*}
By Lipschitz continuity, $\|\nabla F(u_j)\|^2 \leq 2 M \delta F_j$, while the finite variance condition \eqref{eq:finite_variance} implies, via \eqref{eq:finite_variance_classic}, that $\E\left[\sum_{j=1}^{n-1} \|\nabla f_j(u_j)-\nabla F(u_j)\|^2\right]\leq (n-1)\sigma^2$. Hence, since $\sqrt{a+b} \leq \sqrt{a}+\sqrt{b}$ for $a,b\geq 0$, 
\begin{equation}
\label{eq:bnd_eta_recip}
\E\left[\frac{1}{\eta_n}\right]\leq \frac{1}{\eta}\left(\sqrt{b_0^2 + 2(n-1)\sigma^2} + 2\sqrt{M}\E\left[\sqrt{\sum_{j=1}^n \delta F_j}\;\right]\right)
\end{equation}
The bound on $\E\left[\sum_{j=1}^n \eta_j \delta F_j\right]$ can be established, based on the relation 
\begin{align*}
\eta_j \langle\nabla f_j, u_j-u^*\rangle &= \frac{1}{2}\|u_{j+1}-u^*\|^2 - \frac{1}{2}\|u_j-u^*\|^2 + \eta_j^2 \|\nabla f_j\|^2,
\end{align*}
which can be readily derived using the AdaGrad update $u_{j+1}=u_j-\eta_j \nabla f_j$.
Let $\E_j$ denote the conditional expectation with respect to $\xi_j$, given $\xi_1,\hdots,\xi_{j-1}$. Strong convexity and the fact that $u_j$ and $\eta_j$ are independent of $\xi_j$ now imply
\begin{align*}
\E_{j}\left[\eta_j\delta F_j\right] &\leq \eta_j\langle \nabla F_j(u_j), u_j-u^*\rangle = \E_{j}\left[\langle \eta_j \nabla f_j, u_j-u^*\rangle\right] \\
&=\E_j\left[\frac{1}{2}\|u_{j+1}-u^*\|^2 - \frac{1}{2}\|u_j-u^*\|^2 + \eta_j^2 \|\nabla f_j\|^2\right],
\end{align*}
which, by summing over $j=1,\hdots, n$, taking total expectation and recognizing a telescoping sum, yields
\begin{align}
\E\left[\sum_{j=1}^n\eta_j\delta F_j\right] &= \E\left[\sum_{j=1}^n\E_j\left[\eta_j\delta F_j\right]\right] \nonumber \\
&\leq \E\left[\sum_{j=1}^n \left(\frac{1}{2}\|u_{j}-u^*\|^2 - \frac{1}{2}\|u_{j+1}-u^*\|^2 + \frac{1}{2} \eta_j^2\|\nabla f_j\|^2\right)\right]\nonumber\\
&\leq \frac{1}{2}\|u_{1}-u^*\|^2 + \frac{1}{2}\E\left[\sum_{j=1}^n\eta_j^2 \|\nabla f_j\|^2\right].\label{eq:bnd_etadf_etagradf}
\end{align}
An upper bound $\E\left[\sum_{j=1}^n\eta_j^2 \|\nabla f_j\|^2\right]$ can be obtained from the strong finite variance condition \eqref{eq:finite_variance}, as well as Lipschitz continuity, giving
\begin{align*}
\frac{1}{2}\E\left[\sum_{j=1}^n\eta_j^2 \|\nabla f_j\|^2\right]
\leq &\frac{1}{1-\frac{4\eta M}{\sqrt{b_0}}}\frac{2 \eta^2}{b_0}(1+\ln n)\sigma^2 
\\ 
&+ \frac{\eta^2}{1-\frac{4\eta M}{\sqrt{b_0}}}  \ln \left( \sqrt{b_0 +2n \sigma^2} + 2 \sqrt{M} \mathbb{E}\left[\sqrt{\sum_{j=1}^n \delta F_j}\;\right]\right).
\end{align*}
Substituing this inequality into \eqref{eq:bnd_etadf_etagradf} and then inequalities \eqref{eq:bnd_etadf_etagradf} and \eqref{eq:bnd_eta_recip} into \eqref{eq:bnd_df}, one obtains an inequality in $\mathbb{E}\left[\sqrt{\sum_{j=1}^n \delta F_j}\;\right]$, which can be solved (see \cite{Li2019a}, Lemmas 5 and 6), yielding the result.
\end{proof}

\section{Numerical Experiments}\label{section:numerical_experiments}

In this section we perform two numerical experiments to illustrate and showcase the properties of the AdaGrad method for the parabolic optimal control problem \eqref{eq:optimal_control}. Example \ref{example:1} explores the convergence behavior of the algorithm and compares it with that of the traditional SGD method, focusing specifically on the effects of step size and convexity. Example \ref{example:2} applies the method to a simplified thermal regulator problem aimed at maintaining a safe overall temperature      

\begin{example} \label{example:1}
We first consider the parabolic problem given by Equation \eqref{eq:state} with spatial domain $[0,1]$ and terminal time $T=0.2$. Let the initial condition be given by $y_0(x)=x(1-x)$, the forcing term be $g(x)=0$, and the diffusion coefficient be the lognormal random field,
\[
a(x,\omega) = a_{\min} + \exp(\tilde{a}(x,\omega)), \qquad x\in [0,1], \omega \in \Omega,
\]  
where $a_{\min}=0.1$ and $\tilde a(x,\omega)$ is a zero-mean Gaussian random field with a Gaussian covariance kernel $k(x_1,x_2) = \sigma^2\exp\left(-\frac{|x_1-x_2|^2}{2l^2}\right)$ with variance $\sigma^2>0$ and correlation length $l>0$. In our experiments, the field is approximated by the truncated Karhunen-Lo\`eve expansion with 40 terms. The optimal control problem seeks to steer the state $y$ to the zero state $y_d(x,t)=0$ through distributed forcing with a regularization parameter $\alpha=0.1$. Both the state and adjoint equations are approximated by piecewise linear finite elements with 50 sub-intervals in space, and the implicit Euler timestepping scheme with 100 sub-intervals in time.  

To verify that the AdaGrad method outlined by Algorithm \ref{alg:adagrad} converges to an optimal control in mean square, we plot the sampled cost function $f(u_j,\omega)$ as well as its gradient $\nabla f(u_j,\omega)$ against the iteration count $j=1,\hdots 50$ in Figure \ref{fig:ex01_convergence}. We use an initial guess of $u_0(x,t)=2$ and optimization parameters $b_0=0.1$ and $\eta=1$. As is usual in stochastic optimization, the cost initially decreases  much variation during its transient phase (roughly 10 iteration steps here), after which it settles down into a stationary phase. Moreover, the linear decrease in $f$ in the log-log scale during the transient phase implies a convergence rate of $O(\frac{1}{j})$. In light of Remark \ref{rem:markov_inequality}, this is to be expected due to the small variance, as shown in Figure \ref{fig:ex01_state_variance}.

\begin{figure}[ht!]
\centering
\begin{subfigure}[t]{0.48\textwidth}
\centering
\includegraphics[width=\textwidth]{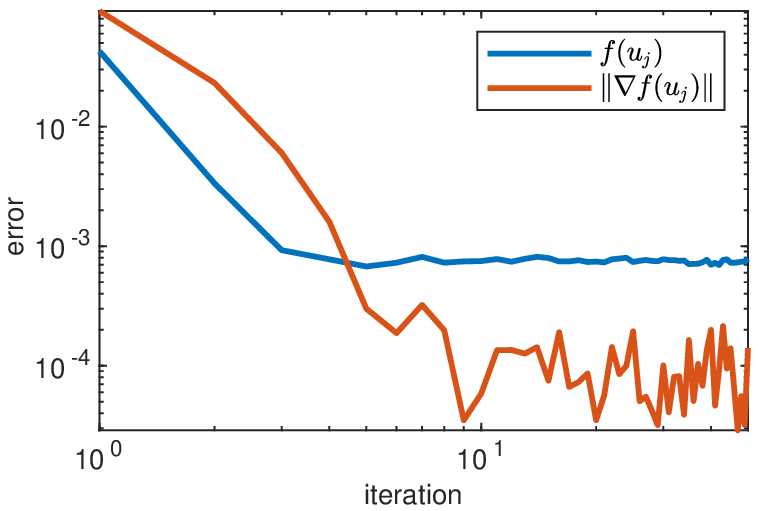}
\caption{Convergence of the cost and gradient norm on a log-log scale.}
\end{subfigure}
\hfill
\begin{subfigure}[t]{0.48\textwidth}
\centering
\includegraphics[width=\textwidth]{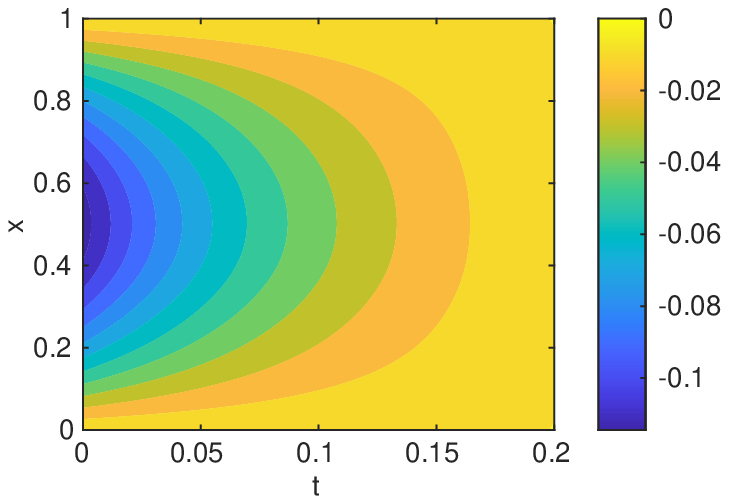}
\caption{Optimal control $u$}
\end{subfigure}
\caption{Convergence of the AdaGrad Method}
\label{fig:ex01_convergence}
\end{figure}

In Figure \ref{fig:ex01_statistics} we plot the sample mean and variance of the state $y$ under the optimal control $u$ to verify that the desired state $y_d(x,t)=0$ is tracked. 

\begin{figure}[ht!]
\centering
\begin{subfigure}[b]{0.48\textwidth}
\centering
\includegraphics[width=\textwidth]{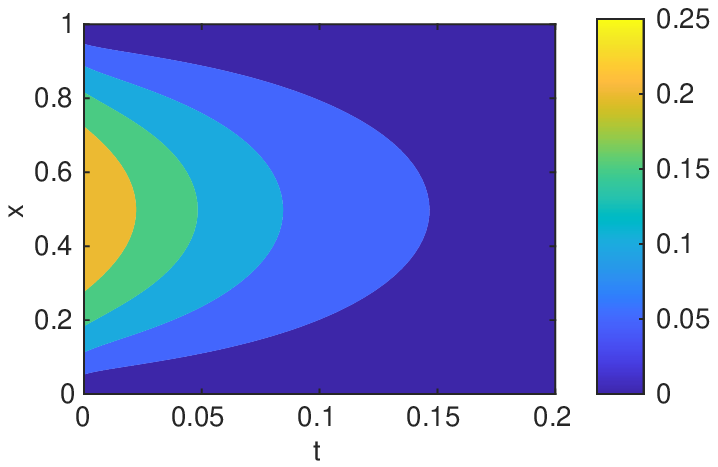}
\caption{The sample mean of the optimal state $y$.}
\end{subfigure}
\begin{subfigure}[b]{0.48\textwidth}
\centering
\includegraphics[width=\textwidth]{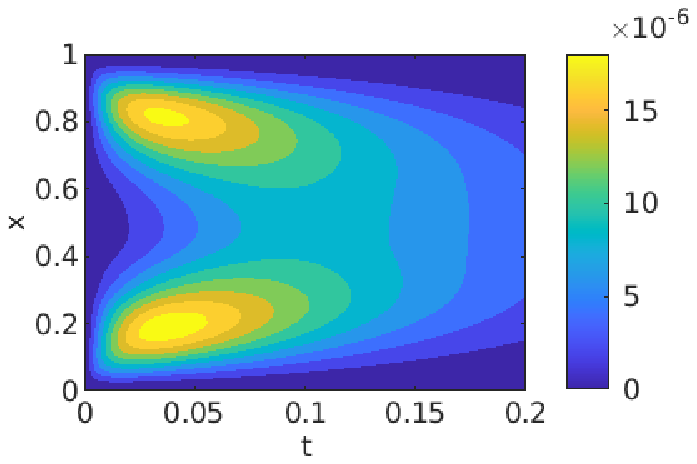}
\caption{The sample variance of the optimal state $y$.}
\label{fig:ex01_state_variance}
\end{subfigure}
\caption{Sample statistics of the optimal state using a sample size of 100.}
\label{fig:ex01_statistics}
\end{figure} 

Next we compare the AdaGrad algorithm's convergence behavior with that of the SGD method. Figure \ref{fig:ex01_step size} shows the effect of the step size on the algorithms' performances. The step size rule for the SGD method is chosen as $\eta_j = \frac{\eta_0}{j+1}$ for $j=0,1,...200$, while that of the AdaGrad algorithm was chosen according to Equation \eqref{eq:adagrad_step size}, with $b_0=1$ to ensure a fair comparison. Evidently, the step size affects the convergence behavior of both algorithms for this problem, with a decrease in $\eta_0$ leading to a deterioration in convergence. However, the AdaGrad method (in black) seems to be less adversely affected than the SGD method (in red). 

\begin{figure}[ht!]
\centering
\includegraphics[width=0.6\textwidth]{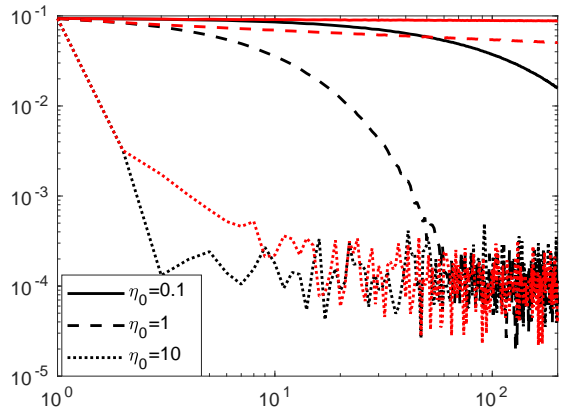}
\caption{Convergence of the gradient norm $\|\nabla f(u_j)\|$ for different initial stepsizes, using either the SGD (red) or AdaGrad (black) method.}
\label{fig:ex01_step size}
\end{figure}

Figure \ref{fig:ex01_convexity} shows how the regularization parameter $\alpha$, which determines the problem's strong convexity, plays a role in the algorithms' convergence behaviors. For both algorithms we set the step size parameters to ensure an initial step size of $\eta_0=10$ and vary the regularization parameter $\alpha$. The plot suggests that, while a reduction in the strong convexity of the problem reduces the convergence rates of both algorithms, the AdaGrad method (in black) is more robust to the loss in convexity than the SGD method (in red), whose convergence breaks down for $\alpha=0.01$.

\begin{figure}[ht!]
\centering
\includegraphics[width=0.6\textwidth]{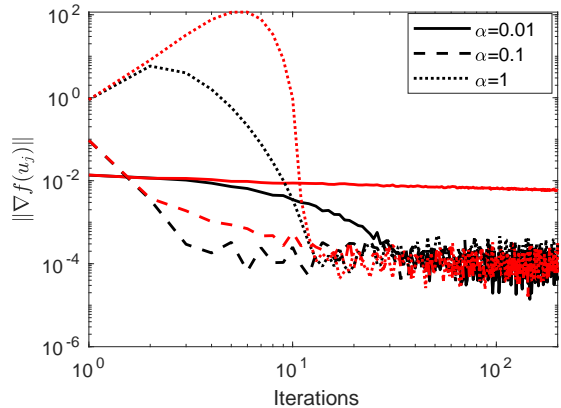}
\caption{Convergence of the gradient norm $\|\nabla f(u_j)\|$ for different regularization parameters in the first 200 iterations of either the SGD (red) or the AdaGrad (black) method with an initial step size of $\eta_0=10$.}
\label{fig:ex01_convexity}
\end{figure}
\end{example}

\begin{example}\label{example:2}
In this example we consider a simple two-dimensional model of a lithium cell whose forcing term represents the heat generation rate associated with a power load  resulting from two successive charge or discharge cycles.

The computational domain $D$, shown in Figure \ref{fig:ex02_domain}, represents a cross-section of a cylindrical cell with inner radius of 4mm, an outer radius of 32mm, and a length of 198mm. For simplicity, the surrounding temperature, as well as that in the axial center is assumed to be constant at $T_o=18^\circ \mathrm{C}$. 
\begin{figure}[ht!]
\centering
\includegraphics[scale=1]{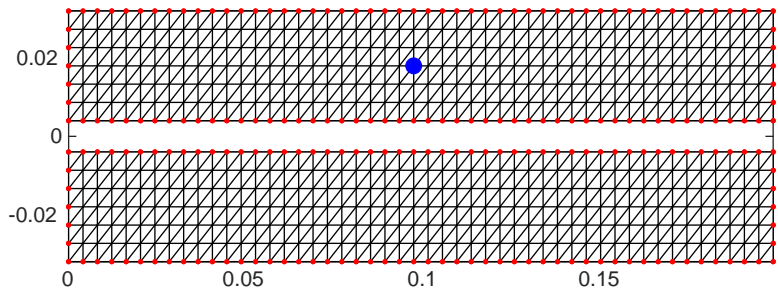}
\caption{Diagram of the computational mesh for the two-dimensional model of the lithium cell.}
\label{fig:ex02_domain}
\end{figure}

The material is treated as a homogeneous solid with density $\rho=2118\; \mathrm{kg} \ \mathrm{m}^{-3}$,  specific heat $c_p=765\; \mathrm{J}\; \mathrm{kg}^{-1}\; \mathrm{K}^{-1}$, and anisotropic thermal conductivities in the horizontal and vertical directions given by $k_{x_1}=66\; \mathrm{W\; m\; K}^{-1}$ and $k_{x_2}=0.66\; \mathrm{W\; m\; K}^{-1}$ respectively (see \cite{Richardson2016}). The equation governing the temperature evolution inside the battery is therefore given by

\begin{equation}\label{eq:lithium_cell_heat}
\begin{cases}
\rho c_p \frac{\partial y}{\partial t} - k_{x_1}\frac{\partial y}{\partial x_1} - k_{x_2}\frac{\partial y}{\partial x_1} = g + u, & x \in D, t>0 \\
y(x,0) = T_o, & x\in D   \\
y(x,t) = T_o, & x\in \partial D, t>0.
\end{cases}
\end{equation}

Heat is generated uniformly over the entire domain, but the timing, duration, and intensity of the power load are assumed to be independently and uniformly distributed uncertain quantities. Specifically, we take the onset time of the first charge pulse to range between 40 and 60 min and that of the second pulse between 200  and 220 min. The duration of each pulse varies between 30 and 60 min, and their intensity lies between 200 and 400 $\mathrm{W}\; \mathrm{m}^{-3}$. Figure \ref{fig:sample_heat_generation_rate} shows 20 sample realizations of the resulting heat generation rate, while Figure \ref{fig:ex02_sample_temperatures} shows the associated uncontrolled temperature profiles at the fixed spatial location $\tilde x = (0.097,0.098)$, indicated on Figure \ref{fig:ex02_domain} by a blue dot.

\begin{figure}[ht!]
\centering
\begin{subfigure}[t]{0.48\textwidth}
\centering
\includegraphics[width=\textwidth]{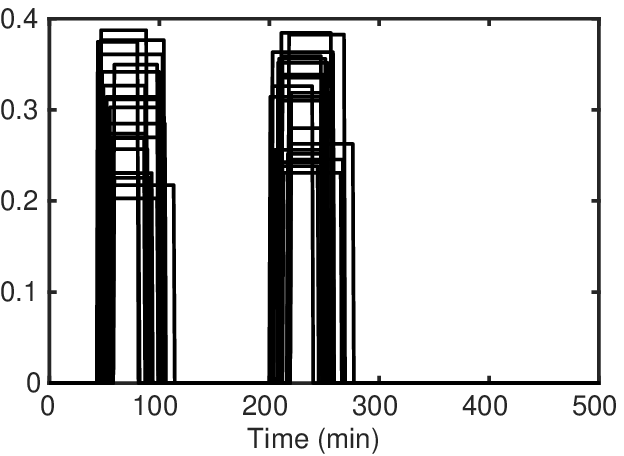}
\caption{Samples of the heat generation rate, scaled by $\rho$, $c_p$.}
\label{fig:sample_heat_generation_rate}
\end{subfigure}
\hfill
\begin{subfigure}[t]{0.48\textwidth}
\centering
\includegraphics[width=\textwidth]{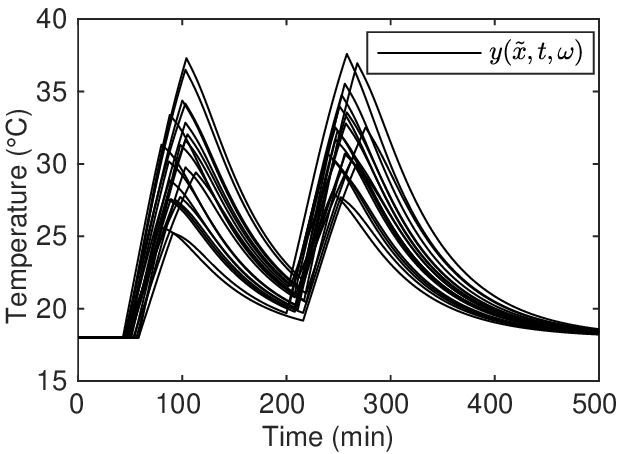}
\caption{Sample paths of the associated temperature with $u=0$ at the point $\tilde x$.}
\label{fig:ex02_sample_temperatures}
\end{subfigure}
\caption{Sample paths of the uncontrolled system.}
\label{fig:ex02_energy}
\end{figure}

We ran 50 iterations of the AdaGrad method with step size parameters $\eta=0.1$ and $b_0=1$, an initial guess of $u_0=0$, and a target state of $y_d=18^\circ\mathrm{C}$. Figures \ref{fig:ex02_convergence_cost} and \ref{fig:ex02_convergence_gradient} show the convergence of the cost functional and of the gradient norm respectively. The persistence of significant random perturbations in the cost indicates far higher levels of variance in this system than were observed in Example \ref{example:1}. 

\begin{figure}[ht!]
\centering
\begin{subfigure}[t]{0.48\textwidth}
\centering
\includegraphics[width=\textwidth]{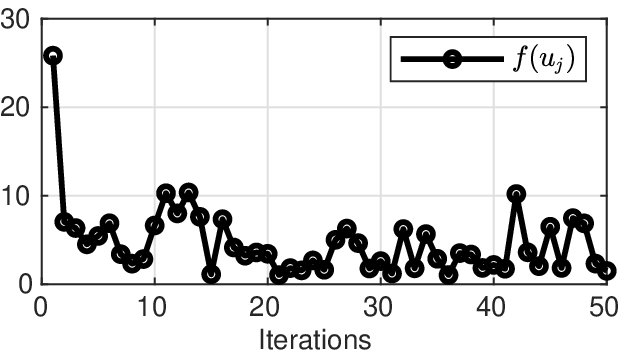}
\caption{Convergence of the cost.}
\label{fig:ex02_convergence_cost}
\end{subfigure}
\hfill
\begin{subfigure}[t]{0.48\textwidth}
\centering
\includegraphics[width=\textwidth]{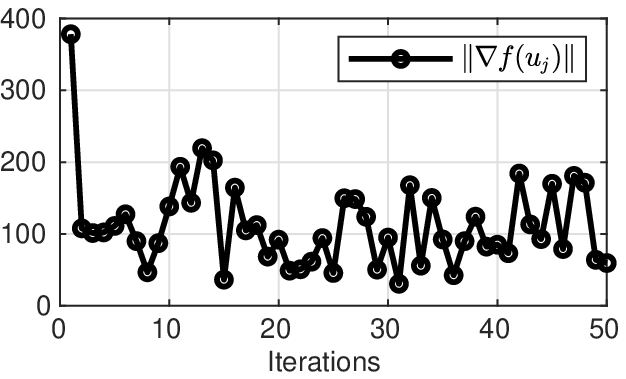}
\caption{Convergence of the gradient norm.}
\label{fig:ex02_convergence_gradient}
\end{subfigure}
\caption{Convergence of the AdaGrad Method for Example \ref{example:2} over 50 iterations.}
\label{fig:ex02_convergence}
\end{figure}

Figure \ref{fig:ex02_controller_convergence} depicts the time evolution of the control iterates, evaluated at the sample point $\tilde x$, and shows how these converge to an optimal control counteracting the heat effect of the two uncertain pulses.  
 
\begin{figure}
\centering 
\includegraphics[scale=1]{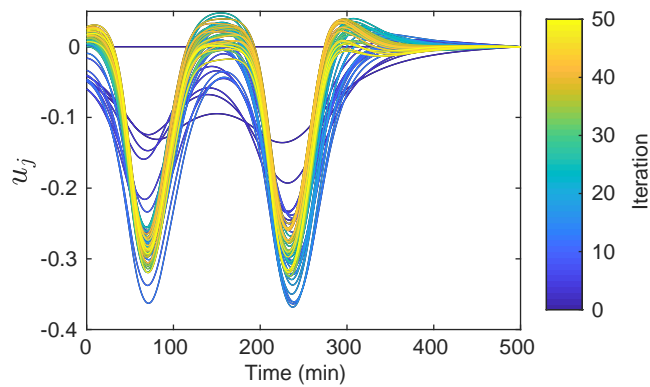}
\caption{Iterates of the control function generated by the AdaGrad method, evaluated at the sample point $\tilde x$.}
\label{fig:ex02_controller_convergence}
\end{figure}

To show that the deterministic controller obtained by this stochastic optimization algorithm leads to a cooling of the system under a variety of conditions, we compare the time evolution of the system's heat energy $E(t)=\int_D y(x,t)^2 dx$ for the initial control (Figure \ref{fig:ex02_energy_before}) and for the optimal control (Figure \ref{fig:ex02_energy_after}). Evidently the addition of the optimal control leads to a significant reduction in heat energy over the entire time period.

\begin{figure}[ht!]
\centering
\begin{subfigure}[t]{0.48\textwidth}
\centering
\includegraphics[width=\textwidth]{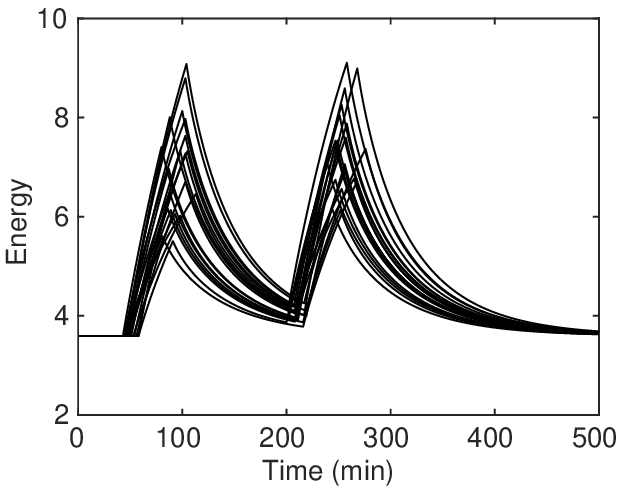}
\caption{Evolution of the heat energy with no control.}
\label{fig:ex02_energy_before}
\end{subfigure}
\hfill
\begin{subfigure}[t]{0.48\textwidth}
\centering
\includegraphics[width=\textwidth]{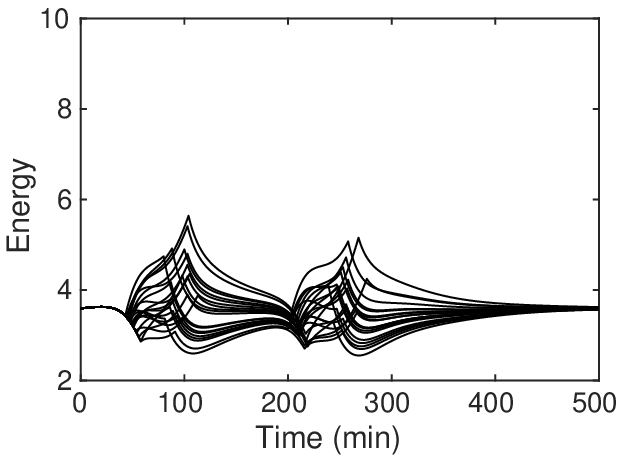}
\caption{Evolution of the heat energy with optimal control.}
\label{fig:ex02_energy_after}
\end{subfigure}
\caption{Comparison of the heat energy over time for the uncontrolled system (left) with that of the controlled one (right).}
\label{fig:ex02_energy}
\end{figure}

\end{example}

\section{Conclusion}\label{section:conclusion}

In this paper we have extended the AdaGrad method to an infinite dimensional optimal control problem for the distributed control of a linear parabolic system. We  related smoothness and finite variance requirements to the statistical distributions of the underlying model parameters and demonstrated how these can be used in thermal regulation of a simple model for an uncertain lithium battery system. It remains to explore how optimal control can be extended to more complex uncertain systems, how this approach can be incorporated into existing battery management systems, and how existing battery degradation models and real-time observers can be used in designing such controls.

\bibliographystyle{plain}
\bibliography{bibliography}
\end{document}